\documentclass[11pt]{amsart}

\usepackage[pdftex]{graphicx} 
\usepackage[usenames,dvipsnames]{xcolor}

\usepackage[english]{babel} 
\usepackage[utf8]{inputenc}

\usepackage[a4paper,margin=2.5cm]{geometry}
\usepackage{upgreek}
\usepackage{amsfonts}
\usepackage{amsmath}
\usepackage{amssymb}
\usepackage{amsthm}
\usepackage{bbm}
\usepackage{paralist}
\usepackage[colorlinks,citecolor=magenta,pagebackref=true,urlcolor=magenta,
pdftex]{hyperref}

\usepackage[T1]{fontenc}

\usepackage{nicefrac}
\usepackage{microtype}
\usepackage{mathrsfs}

\theoremstyle{plain}
\newtheorem{thm}{Theorem}
\newtheorem*{thm*}{Theorem}

\newtheorem{prop}[thm]{Proposition}
\newtheorem{lem}[thm]{Lemma}
\newtheorem{cor}[thm]{Corollary}

\newtheorem*{prb*}{Problem}

\theoremstyle{definition}
\newtheorem{rem}[thm]{Remark}

\theoremstyle{remark}

\newcommand\Defn[1]{\textbf{{#1}}}

\newcommand\mc[1]{\mathcal{#1}}
\newcommand\scr[1]{\mathscr{#1}}

\newcommand\mr[1]{\mathrm{#1}}

\newcommand\mbf[1]{\mathbf{#1}}
\newcommand{\R}{\mathbb{R}}

\newcommand{\Arr}{\scr{A}}
\newcommand\Zono{\mc{Z}}
\newcommand{\rk}{\mr{rk}}

\newcommand\IP[1]{\mathrm{L}_{#1}}

\newcommand\eps{\varepsilon}
\renewcommand\l{\lambda}

\newcommand\Wills{\mathrm{W}}%

\newcommand\codim{\mathrm{codim}}
\newcommand\vol{\mathrm{vol}}

\title[Whitney numbers via measure concentration of intrinsic volumes]%
{Whitney numbers of arrangements via measure concentration of intrinsic
volumes}
\author{Karim A.~Adiprasito}
\address{Einstein Institute for Mathematics, Hebrew University of Jerusalem, Jerusalem, Israel}
\email{adiprssito@math.fu-berlin.de}
\author{Raman Sanyal}
\address{Fachbereich Mathematik und Informatik, %
Freie Universit\"at Berlin, Berlin, %
Germany}
\email{sanyal@math.fu-berlin.de}

\date{\today}


\begin{document}

\begin{abstract}
We verify the Rota--Heron--Welsh conjecture for matroids realizable as
$c$-arrangements: the coefficients of the characteristic polynomial of the
associated matroid are log-concave. This family of matroids strictly contains
that of complex hyperplane arrangements.  Our proof combines the study of
intrinsic volumes of certain extensions of arrangements and the L\'evy--Milman
measure concentration phenomenon on realization spaces of arrangements.
\end{abstract}
\maketitle

\newcommand\Char{\psi}

In generalization of Birkhoff's chromatic polynomial of a
graph~\cite{Birkhoff}, one defines for a matroid $M$ of rank $r$ the
\Defn{characteristic polynomial} 
\begin{equation}
    \chi(M;\lambda) \  := \ \sum_{x\in \IP{M}} \mu(x)\,
    \lambda^{\rk(M)-\rk(x)}
    \ = \ \gamma_0(M) \lambda^r - \gamma_1(M) \lambda^{r-1} + \cdots + (-1)^r
    \gamma_r(M),
\end{equation}
where $\IP{M}$ is the intersection poset or lattice of flats of $M$ with
M\"obius function $\mu(x) = \mu_{\IP{M}}(\widehat{0},x)$ and rank function
$\rk(\cdot)$.  The coefficients $\gamma_{i}(M)$ ---the (unsigned)
\Defn{Whitney numbers of the first kind}--- carry a variety of combinatorial
information of $M$ and have been subject to extensive study; see, for example,
Chapters 7 and 8 of~\cite{white}.  The coefficients $\gamma_i$ coincide with
the Betti numbers of the Orlik-Solomon algebra associated to $M$, and they are
closely related to Milnor numbers and Chern--Schwartz--MacPherson classes of
complements of complex hyperplane arrangements. This paper is devoted to the
following property of characteristic polynomials of matroids:

\Defn{Rota--Heron--Welsh conjecture.} For any matroid $M$, the coefficients of
the characteristic polynomial $\chi(M;\lambda)$ are \Defn{log-concave}, that is,
\[
    \gamma_{i-1}(M) \cdot \gamma_{i+1}(M) \ \le \ \gamma_i(M)^2\quad 
\] 
for all $1 \le i \le n-1$.

By Rota's sign theorem~\cite{Rota}, $\gamma_i(M) > 0$ for all $i$ and hence the
conjecture implies that the sequence of Whitney numbers is unimodal, i.e., 
\[
    \gamma_0 \ \le \  \gamma_1 \  \le \  \cdots  \ \le \  \gamma_{i-1} \ \le \
    \gamma_i \ \ge \ 
    \gamma_{i+1} \  \ge \  \cdots \ \ge \  \gamma_{n-1} \  \ge \  \gamma_n
\]
for some $0 \le i \le n$. Following Aigner~\cite{Aigner}, we define the
\Defn{absolute characteristic polynomial} of $M$
\begin{equation}\label{eqn:achar}
    \Char(M;\lambda)  \ = \ \gamma_0(M) \lambda^r + \gamma_1(M) \lambda^{r-1} + \cdots + \gamma_r(M).
\end{equation}

Spectacular progress towards a resolution of the conjecture has been achieved
by Huh~\cite{Huh} for matroids that can be realized over a field of
characteristic $0$ and in full generality by Adiprasito--Huh--Katz~\cite{AHK}.
The proof in~\cite{AHK} is set in algebraic geometry and the aim of this note
is to prove the following weaker result by appealing to methods from convex
geometry. 

\begin{thm}\label{thm:log_conc2}
    If $M$ is a matroid realizable by a $c$-arrangement, then the sequence of
    Whitney numbers $\gamma_0(M),\gamma_1(M),\dots,\gamma_n(M)$ is
    log-concave.
\end{thm}

Here, a $c$-arrangement is a collection of codimension-$c$ linear subspaces of
$\R^d$ all whose non-empty intersections have codimension divisible by $c$
\cite[Part III]{GM-SMT}. It is easy to check that (rank functions of)
$c$-arrangements give matroids; see Section~\ref{sec:basics}.  For $c=2$,
$c$-arrangements include complex hyperplane arrangements but are strictly more
general (see \cite{Z-DRC}). In particular, there are matroids not realizable
over any field that can be realized as
$c$-arrangements~\cite[Sec.~III.5.2]{GM-SMT}. In this sense,
Theorem~\ref{thm:log_conc2} is not a complete resolution of HRW-conjecture.
For example, it is known that the V\'amos matroid~\cite[Example~2.1.22]{oxley}
does not satisfy Ingleton's inequality \cite{Ingleton} and is therefore not
realizable as a $c$-arrangement \cite{Bjss}.  For some related development,
compare also \cite{A}, where the Lefschetz hyperplane theorem is extended from
the complex-algebraic case to $c$-arrangements.

Whereas Huh's proof is set in algebraic and tropical geometry (see
also~\cite{katz-huh}), our proof is in the realm of classical convex geometry.
The key idea follows a recent geometric approach to the MacPherson conjecture
\cite{A2}:  The main result gives a geometric representation of
the Whitney numbers of a  $c$-arrangement $\Arr$ as the intrinsic volumes of a
high-dimensional convex body. The log-concavity then simply follows from the
Alexandrov--Fenchel inequalities.  To establish this, we first describe what
we call an \emph{extension} of an arrangement (Section~\ref{sec:exts}). This
yields a sequence of probability spaces of arrangements. We prove that the
associated convex bodies (zonotopes for hyperplane arrangements, discotopes
for $c$-arrangements) have a Wills polynomial resembling the characteristic
polynomial of $\Arr$ asymptotically almost surely (a.a.s.) using the most
basic form of L\'evy--Milman measure concentration. The curiosity of this
proof is underscored by the fact that arrangements in general have complicated
realization spaces (see Remark~\ref{Mnev}) but a geometry that nevertheless
allows for a probabilistic treatment.

\textbf{Acknowledgements.} We would like to thank the Miller Institute at UC
Berkeley where this research was initiated. These results where first
presented at the Oberwolfach Workshop \emph{Geometric and Algebraic
Combinatorics} in February 2015 and we thank the participants for helpful
comments. K.~Adiprasito was supported by an EPDI/IPDE postdoctoral fellowship
and a Minerva fellowship of the Max Planck Society, and by the Romanian NASR,
CNCS---UEFISCDI, project PN-II-ID-PCE-2011-3-0533.  R.~Sanyal was supported by
the DFG-Collaborative Research Center, TRR 109 ``Discretization in Geometry
and Dynamics''.

\section{Convex geometry of $c$-arrangements}\label{sec:basics}

In this paper, we focus on matroids that can be realized by some central
$c$-arrangement in $\R^d$. For a general arrangement $\Arr$ of subspaces in
$\R^d$, we write $\IP{\Arr}$ for the intersection poset, that is, the
nonempty intersections of elements in $\Arr$ ordered by reverse inclusion. The
minimum is thus $\R^d$ and if $\Arr$ is central, then the maximum is
$\hat{1} = \bigcap_{H \in \Arr} H$. A central arrangement is \Defn{essential}
if $\hat{1} = \{0\}$.
In analogy to hyperplane arrangements, we define the 
\Defn{absolute characteristic polynomial} of a subspace arrangement $\Arr$ as
\[
    \Char(\Arr; \lambda) \ := \ \sum_{x \in \IP{\Arr}} \mu(x)
    (-1)^{d-\dim(x)}\lambda^{\dim(x)}
\]
where $\mu(x) = \mu(\hat{0},x)$ is the M\"obius function of $\IP{\Arr}$;
see~\cite[Sect.~4.4]{Bjss}.  The arrangement $\Arr$ is a \Defn{$\boldsymbol
c$-arrangement} if all subspaces are of codimension $c$ and $\codim(x)$ is
divisible by $c$ for all $x \in \IP{\Arr}$.  
It was first noted in~\cite{GM-SMT} that for a central $c$-arrangement,
$x \mapsto \frac{1}{c}\codim(x)$ is the rank function of a matroid $M(\Arr)$.
    
For an element $H \in \Arr$, we define the \Defn{deletion} 
and the \Defn{contraction}
\[
    \Arr \backslash H \ := \ \{ H'  \in \Arr : H' \not\subseteq H\}
    \qquad \text{and} \qquad
    \Arr / H \ := \ \{ H' \cap H : H' \in \Arr \backslash H\}.
\]
The absolute characteristic polynomial satisfies the deletion-contraction
identity
\begin{equation}\label{eqn:char_del_contr}
    \Char(\Arr; \lambda) \ = \ \Char(\Arr \backslash H;\lambda) \ + \ \Char(\Arr /
    H;\lambda).
\end{equation}
If $\Arr$ is a central $c$-arrangement in $\R^d$ realizing a matroid $M$ of
rank $r$, then $\Char(\Arr;\lambda)= \lambda^{d-r}\Char(M_\Arr;\lambda^{c})$.
We refer the reader to Stanley's lecture notes on hyperplane
arrangements~\cite{StanleyHyp} (see also~\cite{StanleyEC1new}) and Bj\"orner's
excellent treatment of subspace arrangements~\cite{Bjss}.

\subsection{Zonotopes and discotopes} 
\label{ssec:zono} %

We denote by $\kappa_d = \vol_d(B^d) = \frac{\pi^{n/2}}{\Gamma(n/2+1)}$ the
volume of the unit $d$-ball. For an affine  subspace $H \subset \R^d$ of
dimension $k$, let us write $H^\perp$ for $(d-k)$-dimensional linear subspace
orthogonal to $-p + H $ for $p \in H$.  We write $\mbf{n}_H :=
\kappa_{d-k}^{-(d-k)}B^d \cap H^\perp$ to denote the ball in $H^\perp$ of
volume $1$.  To a subspace arrangement $\Arr=\{H_1, H_2, \dots, H_n\}$ in
$\R^d$ we associate the convex body
\[
    \mc{Z}(\Arr) \ := \ \mbf{n}_1 + \mbf{n}_2 + \cdots + \mbf{n}_n.
\]
If $\Arr$ is a hyperplane arrangement, then $\mc{Z}(\Arr)$ is the zonotope
corresponding to the unit normals to the hyperplanes in $\Arr$. For subspace
arrangements this is more general and $\mc{Z}(\Arr)$ is called the
\Defn{discotope} of $\Arr$. In analogy, we call $\mbf{n}_H$ the
\Defn{generalized (unit) normal} of $H$ or the \Defn{$k$-normal} if we want to
emphasize the dimension $k = \dim H$.

\subsection{Wills polynomials and normal cones}
Let $K$ denote any closed $r$-dimensional convex body in $\R^d$ and let $B^d$
be the unit ball.  Steiner's formula asserts that the volume of the Minkowski
sum of $K$ and the dilated ball $\lambda B^d$ is given by
\begin{equation}\label{SteinerP}
    \vol_d(K+\lambda B^d) \ = \ 
    \nu_d(K) \kappa_{0} +
    \nu_{d-1}(K) \kappa_{1} \lambda^{1} + 
    \cdots +
    \nu_0(K) \kappa_{d} \lambda^{d}.
\end{equation}
This is called the \Defn{Steiner polynomial} of $K$.  The coefficients
$\nu_i(K)$, called the \Defn{intrinsic volumes} of $K$, will be of great
importance to us. For a polytope $P \subset \R^d$, they have a simple
interpretation: For a face $F \subseteq P$ of dimension $k$, let
\[
    N_F(P) \ := \ \{ \omega \in \R^d : \omega^tx \le \omega^ty, x \in P, y
    \in F\}
\]
be the \Defn{normal cone} of $F$ at $P$ and define the \Defn{external angle} of
$F$ at $P$ as 
\[
    \alpha_F (P) \ := \ \frac{\vol_{d-k}(N_F(P) \cap
    B^d)}{\kappa_{d-k}}.
\]
The intrinsic volumes of $P$ can now be expressed as
\begin{equation}\label{SteinerL}
    \nu_i(P) \ = \ \sum_{F\ i\text{-face of}\ P}
    \alpha_F(P)\cdot\vol_i(F).
\end{equation}

A central result concerning the coefficients of Steiner polynomials is the following
consequence of the \mbox{Alexandrov--Fenchel} inequalities (cf.\ \cite{Schneider93}).
\begin{thm}\label{AF}
    The coefficients $(\nu_i(K) \kappa_{d-i})_{i=0,\dots,d}$ of the
    $d$-dimensional Steiner polynomial of a $r$-dimensional convex body, $r\le
    d$, form a log-concave sequence.
\end{thm}

It is clear that the Steiner polynomial makes reference to the ambient space
whereas the intrinsic volumes do not. This leads to the so-called \Defn{Wills
polynomial}~\cite{wills,hadwiger}: For a $d$-dimensional convex body $K$ we
define
\begin{equation}
    \Wills(K;\lambda) \ := \ %
    \nu_d(K) +
    \nu_{d-1}(K)  \lambda^{1} + 
    \cdots +
    \nu_0(K)  \lambda^{d}.
\end{equation}

For a zonotope, the Wills polynomial carries quite some combinatorial
information: Let $Z = \sum_{i=1}^n [-z_i,z_i]$ be a zonotope.
The $k$-faces of a zonotope $Z$ can be grouped in \emph{belts}.
Any two $k$-faces $F_1,F_2 \subset Z$ in the same belt are translates and
hence $\vol_k(F_1) = \vol_k(F_2)$. The belts of $Z$ are in bijection with the
flats of the corresponding hyperplane arrangement $\Arr$.  Moreover, the sum
of the external angles of all $k$-faces in a belt sums to $1$ and therefore
\begin{equation}\label{eqn:wills_zono}
\Wills(Z;\lambda) \ =\ \sum_{L \in \mc{L}(\Arr)}
    \vol_{\dim F_L}(F)\, \lambda^{d-\dim F_L}
\end{equation}
where $F_L$ is a representative of a face of the belt corresponding to $L$.
Let $[-z_i,z_i]$ be a generating segment of $Z$ and denote by $Z\backslash i$
the deletion and by $Z/i$ the contraction (i.e., projection onto $z_i^\perp$),
then
\begin{equation}\label{eqn:wills_del_contract}
    \Wills(Z;\lambda) \ = \ \Wills(Z\backslash i;\lambda) \ + \ \|z_i\|
    \Wills(Z / i;\lambda).
\end{equation}

As an example, let $\Arr_d$ be the arrangement of the $d$ coordinate
hyperplanes in $\R^d$. The corresponding zonotope $Z_d$ is a translate of the
unit cube $[0,1]^d$. Hence
\[
    \Wills(Z_d;\lambda) \ = \ (1+\lambda)^d \ = \ \sum_{i=0}^d \binom{d}{i}
    \, \lambda^{d-i}.
\]

Observe that $\Wills(Z_d;\lambda) = \Char(\Arr_d;\lambda)$.  It is natural to
attempt to find a convex body $K$ for every matroid $M$ such that
$\Wills(K;\lambda)\ =\ \Char(M;\lambda)$.  On second thought, this is likely
to fail, since the Wills polynomial encodes geometric information rather than
combinatorial and for general zonotopes do not satisfy the appropriate
deletion-contraction recurrence. Repairing these defects will be the purpose
of this note.

It was shown by McMullen~\cite{McMullenIV} that the intrinsic volumes are also
log-concave. 
\begin{cor}\label{cor:wills_log}
    For a $d$-dimensional convex body $K$, the coefficients of the Wills
    polynomial $\nu_0(K),\dots,\nu_d(K)$ form a log-concave sequence.
\end{cor}
We repeat the proof since it fits perfectly into our setting.
\begin{proof}
    Let $K \subset \R^d$ be a $d$-dimensional convex body. For every $n \ge
    d$, the have an isometric embedding $K \subset \R^n$. Thus, we can
    consider the coefficients of the Steiner polynomials
    $S_n(K;\lambda) = \vol_n(K + \lambda B_n)$ for $n \rightarrow \infty$.
    Since the sequence $\nu_i(K) \kappa_{n-i}$ is log-concave, so is the
    sequence
    \[
        \widetilde{\nu}_{i,n}(K) \:= \ \nu_i(K)\cdot\kappa_{n-i}\cdot
        \pi^{-\frac{n-i}{2}} \sqrt{\pi {n}} (\tfrac{n}{2e})^{\frac{n}{2}}
    \]
    By the first Stirling formula, we infer that $\widetilde{\nu}_{i,n}(K)
    \xrightarrow{\ n \longrightarrow \infty\ }\nu_i(K)$. 
\end{proof}

\subsection{Measure concentration} The philosophy of measure concentration
makes our use of this principle quite clear: If $X$ is a random variable in a
metric probability space depending on sufficiently many, sufficiently
independent variables then $X$ is virtually constant. We argue here that if the
normals generating an arrangement are sufficiently independent, then (the
Wills polynomial of) a random arrangement is essentially independent of the
realization, and hence ``combinatorial''.  We refer to \cite{GM} and
\cite{MS} for the necessary background.

\newcommand\Gr{\mathrm{Gr}}
The underlying principle of measure concentration is geometric, and in this
context goes back to L\'evy and later Milman (cf.\ \cite{GM}), who revealed
the connection to isoperimetric properties.  Ultimately, we shall only need a
very special case of this technology: $S^d$ with the natural angular distance
$\updelta$ and uniform distribution $\upmu$ for hyperplane arrangements, and,
more generally, the Grassmannians $\Gr_{r,d}$ with the uniform measure $\upmu$
and metric $\updelta$ defined as the Hausdorff distance between unit balls.
For a subset $A \subset G_{r,d}$, we denote by $A_\eps = \{ x \in
\Gr_{r,d} : \updelta(x,A) < \eps\}$ the $\eps$-neighborhood of
$A$.

\begin{prop}[cf.\ {\cite[Sec.\ 6.6]{MS}}]\label{prp:levy}
    The space of $(\Gr_{r,d},\updelta,\upmu)$ is a \Defn{normal L\'evy family}
    (w.r.t.\ $d$), i.e.\ for every Borel subset $A\subset \Gr_{r,d}$ with
    $\upmu(A) = 1/2$, we have
    \[
        \mu(A_{\eps})\ =\ 1-\sqrt{\frac{\pi}{8}} \cdot e^{-
        \frac{1}{8} d\eps^2}
    \]
    for all $\eps > 0$.
\end{prop}
For the use of this proposition, note that for a metric probability space
$\mc{X}=(X,\delta,\mu)$ with 
\[
    \alpha_{\mc{X}}(\eps)\ :=\ 1-\inf\left\{\mu(A_{\eps}): A \subseteq X \
    \text{Borel},\ \mu(A)\ge \tfrac{1}{2}\right\},
\]
a $r$-Lipschitz function $f$ on $X$ satisfies
\[
    \mu(|f(x)-M_f|>\eps) \ \le \
    2\alpha_{\mc{X}}(\tfrac \eps r),
\]
for $M_f$ the L\'evy mean of $f$. Recall that the L\'evy mean $M_f$ satisfies
\[
\mu(f(x) \le  M_f) \ge \frac 1 2 \quad \text{and} \quad
\mu(f(x) \ge  M_f) \ge \frac 1 2.
\]

Let us mention another feature of measure
concentration on $S^d$ (and the Grassmannian): Consider $A^{d,k}$ any
$(d-k)$-dimensional totally geodesic subspace of $S^d$, endowed with its
natural intrinsic metric and uniform measure $\upmu$. Then there are uniform
constants $\widetilde{C}_k$, $\widetilde{c}_k>0$ such that
\begin{equation}\label{eq:orthogonal}
    \upmu(A^{d,k}_\eps) \ = \ 1-\widetilde{C}_k e^{-
    \widetilde{c}_k \cdot d\cdot \eps^2} .
\end{equation}

In addition to measure concentration, this inequality makes clear that if $A$
is a totally geodesic subspace of small dimension in $S^d$, then most of the
measure lies in the orthogonal complement to $A$. 

\section{Extensions of arrangements and Wills polynomials of the L\'evy mean}\label{sec:exts}

In this section we construct for every $c$-arrangement $\Arr$ a parametrized
family of arrangements. Viewed as a probability space, we can use measure
concentration to verify that the Wills polynomials corresponding to the
associated discotopes satisfy the deletion-contraction property of
characteristic polynomials and asymptotically almost surely coincide with
them. Ultimately, the proof of Theorem~\ref{thm:log_conc2} is probabilistic but
the intuition of measure concentration allows for a simple enough explanation.

\subsection{Uniform matroids -- an illustration}

The general philosophy of the proof is simple: Consider a uniformly
distributed collection of $n$ random vectors in $S^{d-1}\subset\R^d$ for $n <
d$. Let $\Zono_{n,d}$ be the corresponding probability space of zonotopes.
What is the Wills polynomial of a \emph{typical} zonotope in $\Zono_{n,d}$?

Clearly, for $d \gg 0$ large, measure concentration dictates that
$\Wills(Z;\lambda)$ for $Z\in \Zono_{n,d}$ almost surely equals the Wills
polynomial of the L\'evy mean, denoted by $\Wills(n,d;\lambda)$.  Moreover,
for $d \rightarrow \infty$, the random vectors are essentially orthogonal to
one another and the geometric
deletion-contraction~\eqref{eqn:wills_del_contract} of Wills polynomials
yields
\[
    \Wills(n,d;\lambda)  \ \asymp \ \Wills(n-1,d;\lambda) \ + \
    \Wills(n-1,d-1;\lambda)
\]
where $f\asymp g \ :\Leftrightarrow\ |f-g| \xrightarrow{\, d\rightarrow
\infty\,} 0$. For $n=d$ or $d = 1$ it is easy to verify that
$\Wills(n,d;\lambda) = (1+\lambda)^d$. 

The matroid corresponding to $n < d$ general vectors in $S^{d-1}$ is
independent of the chosen vectors and is the uniform matroid $U_{n,n}$.
Inspecting its characteristic polynomial now reveals that asymptotically
almost surely, $\Wills(n,d;\lambda) = \Char(U_{n,n};\lambda)$.  Log-concavity
of $\Char(U_{n,d};\lambda)$ then follows from Corollary~\ref{cor:wills_log}.

This example illustrates the underlying idea of our proof but also pinpoints
the obstacles that need to be overcome: The colinearities encoded by a typical
matroid prevent an associated zonotope from being random.  Consider the
uniform matroid $U_{2,3}$ on three elements with rank $2$. We realize it in
$\R^{d+1}$ by choosing two unit vectors $x, y$ uniformly at random in $S^d$,
and a third unit vector $z$ uniformly at random in their common span. Then $x$
and $y$ are almost orthogonal, but $z$ is not (since it is not sufficiently
independent), so the Wills polynomial of a random zonotope does not
concentrate.

To treat this problem, we rely on a extension operation, but one that changes
the matroid to a more ``flexible'' matroid. Nevertheless, the original
information shall not be lost completely.

\subsection{Extensions of arrangements and characteristic
polynomials}\label{ssec:extensions}

We consider three extension constructions for subspace arrangements.

\newcommand\Triv[1]{\mathsf{T}_{#1}}
\Defn{The trivial extension.} The trivial extension of an arrangement was
already implicitly used in the proof of Corollary~\ref{cor:wills_log}.  For an
arrangement $\Arr = \{ H_i \subset \R^d : i = 1,\dots,n \}$, the trivial
extension is
\[
\Triv{\ell}(\Arr) := \{ H_i \times \R^\ell \subset \R^{d + \ell} : i
    =1,\dots,n \}. 
\]
The intersection poset of $\Arr$ is unchanged but the dimension of every
element increases by $\ell$. In particular, for the characteristic polynomial
we have
\[
    \Char(\Triv{\ell}(\Arr); \lambda ) \ = \ \lambda^\ell \Char(\Arr;\lambda).
\]

The next two extensions depend the choice of generic subspaces and hence
produce a parametrized collection of arrangements.

\newcommand\LPE[1]{\mathsf{Pr}_{#1}}%
\newcommand\RP{\R\mathrm{P}}%
\Defn{The large product extension.} Let $\Arr$ be a $c$-arrangement in $\R^d$
and let $k,h \ge 1$ be fixed parameters. Let $\Arr' = \Triv{k}(\Arr)$ 
be the trivial extension to $\R^{d+k}$. Choose $k$ general directions
$s_1,\dots,s_k \in \RP^{d+k-1}$, called the \Defn{extension
directions}. For every $s_i$, let $(S_{i,j})_{j=1,\dots,h} \subset \R^{d+k}$
be distinct affine hyperplanes parallel to $s_i^\perp$. The
\Defn{large product extension} with respect to $k$ and $h$ is defined as
\[
    \LPE{k,h}(\Arr) \ := \ \Arr' \cup \{S_{ij} : i=1,\dots,k, j = 1,\dots,h\}.
\]
Note that $\LPE{k,h}(\Arr)$ is not central in general and a $c$-arrangement
only when $c=1$. The generalized normal $\mbf{n}_i = \frac{1}{2}(s_i \cap B^d)$
corresponding to $s_i^\perp$ is called the \Defn{extension direction} and
$(S_{i,j})_{j}$ are the \Defn{extension hyperplanes}.  For fixed $\Arr$, this
construction yields a collection of arrangements parametrized by
$(\RP^{d + k -1})^k$.
The characteristic polynomial is readily available as follows.

\begin{lem}\label{lem:char_large_prod}
    For a $c$-arrangement $\Arr$ and parameters $k,h \ge 1$
    \[
        \Char({\LPE{k,h}(\Arr)};\lambda)\ = \ (\lambda+h)\cdot
        \Char(\LPE{k-1,h}(\Arr);\lambda). 
    \]
\end{lem}
\begin{proof}
Observe that
\[
    \LPE{k,h}(\Arr / S_{ij}) \ \cong \ \LPE{k-1,h}(\Arr).
\]
Iterating~\eqref{eqn:char_del_contr} yields the claim.
\end{proof}

Hence, we can recover the characteristic polynomial of $\Arr$ as
\[
        \Char(\Arr;\lambda) \ = \ 
    \lim_{h \rightarrow \infty}
    \frac{\Char({\LPE{k,h}(\Arr)};\lambda)}{h^k}.
\]
The large product extension can be further augmented by a trivial extension,
and we abbreviate
$\LPE{k,h,\ell}(\Arr):=\Triv{\ell}(\LPE{k,h}(\Arr))$.

\newcommand\SFE[1]{\mathsf{Sf}_{#1}}
\Defn{The semiflexible extension.}
Let $\Arr$ be an arrangement with a distinguished element $H_e \in \Arr$ and
generalized normal $\mbf{n}_e$. For parameters $k,h \ge 1$, the semiflexible
extension $\SFE{k,h}(\Arr, e)$ is obtained from the large
product extension $\LPE{k,h}(\Arr)$ as follows:
\[
\SFE{k,h}(\Arr, e) \ := \ (\LPE{k,h}(\Arr) \setminus \{H_e\}) \cup
\{H_{e'}\}
\]
where $H_{e'}$ is a linear subspace of dimension $\dim H_e$ whose generalized normal is in general
position in $\mbf{n}_e + \sum_{i=1}^k \mbf{n}_i$, where
$\mbf{n}_1,\dots,\mbf{n}_k$ are the
extension normals.  The element $e'$ is called the \Defn{semiflexible element}
of the extension.

\begin{lem}\label{lem:flexible_extensions}
Let $\Arr$ be an arrangement with distinguished element $e$. Then
\begin{align*}
\Char(\SFE{k,h}(\Arr,e);\lambda) & \ = \
 h \cdot \Char(\SFE{k-1,h}(\Arr,e);\lambda)  \ + \
    \Char({\LPE{k-1,h}(\Arr/e)};\lambda) \ + \
    \Char({\LPE{k-1,h}(\Arr \backslash e)};\lambda).
\end{align*}
\end{lem}

\begin{proof}
    For an extension hyperplane $S$ of $\SFE{k,h}(\Arr,e)$, we note that
    \[
        \SFE{k,h}(\Arr,e) / S \ \cong \ \SFE{k-1,h}(\Arr,e).
    \]
    Hence, iterating~\eqref{eqn:char_del_contr} for all $h$ extension
    hyperplanes of $s$ yields
    \[
        \Char(\SFE{k,h}(\Arr,e);\l) \ = \ h \Char(\SFE{k-1,h}(\Arr,e);\l) +
        \Char(\Arr';\l).
    \]
    Now, $\Arr'$ is an arrangement in $\R^{d+k}$ with the distinguished
    subspace $H_{e'}$. The deletion of $H_{e'}$ results in an arrangement
    $\LPE{k-1,h}(\Arr \setminus e)$ but embedded in $\R^{d+k}$.  Since
    $H_{e'}$ is in general position to the other subspaces, it follows that
    the restriction to $H_{e'}$ yields $\LPE{k-1,h}(\Arr / e)$ from which the
    claim follows.
\end{proof}

Combining Lemma~\ref{lem:char_large_prod} and
Lemma~\ref{lem:flexible_extensions} yields the following.
\begin{cor}\label{cor:flex_limit}
    \[
        \lim_{h\rightarrow \infty}\ \frac{\Char(\SFE{k,h}(\Arr,e);\lambda)}{h^k}
        = \lim_{h\rightarrow \infty}\
        \frac{\Char({\LPE{k,h}(\Arr)};\lambda)}{h^k} = \Char(\Arr;\lambda).
    \]
\end{cor}

The semiflexible extension can be further augmented by a trivial extension,
and we abbreviate $\SFE{k,h,\ell}(\Arr,e) := \Triv{\ell}(
\SFE{k,h}(\Arr,e))$.

\begin{rem}\label{Mnev}
The realization space of a subspace arrangement is defined as the space of
coordinatizations, within respective Grassmannians, modulo affine
transformations. The extension of an arrangement and the original arrangement
have homotopy equivalent realization spaces almost surely.  In fact, it is not
hard to check that the realization space of a large product extension and the
semiflexible extension is stably equivalent to the realization space of the
original arrangement in the sense of Mn{\"e}v \cite{MnevRoklin}. 

This allows us to make an interesting philosophical observation:
\emph{Mn\"ev Universality} and its refinement by
Vakil~\cite{vakil,vakil2} and Kapovich--Millson~\cite{KJ} asserts that
realization spaces of arrangements can be arbitrarily complicated.  Hence,
\emph{topologically} realization spaces of extensions behave badly.  On the
other hand, measure concentration is unaffected by these topological
pathologies as asymptotically this influence vanishes.  
\end{rem}

\subsection{Pushforward measures on arrangement extensions}
We are now interested in the effect of large product and semiflexible
extensions on the Wills polynomial. The trivial extension only increases the
ambient dimension and hence leaves the Wills polynomial unaffected.
Throughout this section let $\Arr$ be a fixed (linear) $c$-arrangement in
$\R^d$ with elements labelled $e_1,\dots,e_n$. For fixed $k,h,\ell$ define
\[
    \SFE{k,h,\ell}(\Arr,e_1,\dots,e_n) \ := \
    \SFE{k,h,\ell}(\SFE{k,h,\ell}(\Arr,e_1,\dots,e_{n-1}),e_n)
\]
with $\SFE{k,h,\ell}(\Arr,e_1)$ as defined in Section~\ref{ssec:extensions}.
This is an arrangement of $n + n \cdot k\cdot h$ subspaces in a Euclidean
space of dimension $d + n \cdot (k+\ell)$. For every element $e_i$ there is a
corresponding semiflexible element $e'_i$. More precisely,
$\SFE{k,h,\ell}(\Arr,e_1,\dots,e_n)$ is a collection of arrangements
parametrized by 
\begin{equation}\label{eqn:param_space}
    (\RP^{d+ n(k+\ell) -1})^{kn} \times (\Gr_{k+c,c})^n
\end{equation}
corresponding to the choice of $kn$ (general) extension directions and $n$
semiflexible elements. Note that for chosen extension directions
$s_1,\dots,s_k$, a semiflexible element $H_e'$ for the codimension-$c$
subspace $H_e$ corresponds to the choice of a $c$-dimensional subspace in
$H_e^\perp + \mathrm{span}\{s_1,\dots,s_k\} \cong \R^{k+c}$. The particular
choice of the extension hyperplanes is irrelevant for our purpose.

The uniform measure on~\eqref{eqn:param_space} makes
$\SFE{k,h,\ell}(\Arr,e_1,\dots,e_n)$ into a probability space and the Wills
polynomial of the discotope corresponding to $\Arr' \sim
\SFE{k,h,\ell}(\Arr,e_1,\dots,e_n)$ is a random variable.  We note the
following consequence of deleting, respectively contracting the semiflexible
element $e'_n$:
\begin{align}\label{eqn:SFE_del}
    \SFE{k,h,\ell}(\Arr,e_1,\dots,e_n) \backslash e'_n &\ \cong \ 
    \LPE{k,h,\ell}(\SFE{k-1,h,\ell}(\Arr \setminus e_n ,e_1,\dots,e_{n-1})) \\
    \intertext{and}
\label{eqn:SFE_contr}
    \SFE{k,h,\ell}(\Arr,e_1,\dots,e_n) / e'_n &\ \cong \ 
        \LPE{k,h,\ell}(\SFE{k-1,h,\ell}(\Arr / e_n ,e_1,\dots,e_{n-1})).
\end{align}
The corresponding maps
\begin{align*}
    (\RP^{d+ n(k+\ell) -1})^{kn} \times (\Gr_{k+c,c})^n
    &\ \rightarrow \
    (\RP^{d+ n(k+\ell) -1})^{kn} \times (\Gr_{k+c,c})^{n-1} \; \text{ and }\\
    (\RP^{d+ n(k+\ell) -1})^{kn} \times (\Gr_{k+c,c})^n
    &\ \rightarrow \
    (\RP^{d+ n(k+\ell) -2})^{kn} \times (\Gr_{k+c,c})^{n-1}
\end{align*}
yield a pushforward of the uniform measure that will be utilized in the proof
of the following result.

For a polynomial $p(\lambda) = \sum_i a_i \lambda^i$, let us denote by
$[p(\lambda)]_i = a_i$
the coefficient of $\lambda^i$. We also abbreviate $\Wills(\Arr;\lambda) =
\Wills(\mc{Z}(\Arr);\lambda)$.

\begin{thm}\label{thm:measure1}
    Let $\Arr$ be a $c$-arrangement on $n$ elements. For sufficiently fast
    growing sequences $(h_k)_{k}$ and $(\ell_k)_k$ the following holds
    asymptotically almost surely for $k \rightarrow \infty$
    \[
    h_k^{-kn}\, \left[\Wills(\Arr';\lambda) \right]_{i\cdot c}
        \ \asymp\ \
        h_k^{-k(n-1)}\, \left[ 
        \Wills(\Arr'';\lambda)+\Wills(\Arr''';\lambda) \right]_{i\cdot c}
    \]        
    for all $0 \le i \le \frac{\rk{\Arr}}{c}$ and where where $\Arr' \sim
    \SFE{k,h_k,\ell_k}(\Arr,e_1,\dots,e_n)$, $\Arr'' \sim
    \SFE{k,h_k,\ell_k}(\Arr \setminus e_n, e_1,\dots,e_{n-1})$, and $\Arr'''
    \sim \SFE{k,h_k,\ell_k}(\Arr/e_n,e_1,\dots,e_{n-1})$.
\end{thm}

\begin{proof}
    Notice that if $\ell_k$ is a sequence of positive integers large enough
    with respect to $k$, then $\SFE{k,h,\ell_k}(\Arr,e_1,\dots,e_n)$ is a
    normal L\'evy family following Proposition~\ref{prp:levy} (independent of
    the value of $h$). 

    The intrinsic volumes of the associated discotopes are Lipschitz
    continuous functions on the parameter space of
    $\SFE{k,h,\ell_k}(\Arr,e_1,\dots,e_n)$ with Lipschitz constants depending
    on $h$ for the extension hyperplanes and on $\Arr$ for all other elements.
    It follows that if $\ell_k$ is large enough with respect to $h_k$, then
    $\Wills(\cdot;\lambda)$ converges to the L\'evy mean asymptotically almost
    surely. We may therefore treat the Wills polynomial of $\Arr' \sim
    \SFE{k,h_k,\ell_k}(\Arr,e_1,\dots,e_n)$ as virtually constant.

    The specific geometry of the Grassmannian stronger dictates that the
    normal $\mbf{n}_{e'_n}$ is a.a.s.\ orthogonal to all other elements of the
    arrangement $\Arr'$.  Hence, for a random element $\Arr'$
    \[
        \left[ \Wills(\Arr';\lambda) \right]_{i\cdot c} \ \asymp \ 
        \left[ \Wills(\Arr'\setminus e'_n;\lambda) \ + \
        \Wills(\Arr / e'_n;\lambda) \right]_{i \cdot c}
    \]
    for all $0\le  i \le \frac{\rk(\Arr)}{c}$ by choice of normalization.  
    By~\eqref{eqn:SFE_del} and~\eqref{eqn:SFE_contr}, we may further
    approximate
    \begin{align*}
        h_k^{-kn}\,\Wills(\Arr' \setminus e_n;\lambda) &\ \asymp \
        h_k^{-kn}\, \Wills(\scr{B};\l) &\text{with }  
        \scr{B} \sim \LPE{k,h_k,\ell_k} \SFE{k,h_k,\ell_k} (\Arr'
        {\setminus} e_n,e_1,\dots,e_{n-1}) \\
        \intertext{ and}
        h_k^{-kn}\, \Wills(\Arr/e_n;\lambda) &\ \asymp \ 
        h_k^{-kn}\, \Wills(\scr{C};\l) &\text{with } \scr{C} \sim
        \LPE{k,h_k,\ell_k} \SFE{k,h_k,\ell_k} (\Arr'/
        e_n,e_1,\dots,e_{n-1}).
    \end{align*}
    as $k \rightarrow \infty$. 
    Finally, observe that the asymptotic effect of a large product extension
    $\LPE{k,h_k,\ell_k}$ on $\Wills$ is a multiplication of the Wills
    polynomial by $(h_k)^k$ (within a constant error term).
\end{proof}

The relation between coefficients of Wills polynomials yields our main result.

\begin{thm}\label{thm:main}
    Let $\Arr$ be a $c$-arrangement with elements $e_1,\dots,e_n$ and
    $(h_k)_k, (\ell_k)_k$ sufficiently fast growing sequences. For $k
    \rightarrow \infty$ asymptotically almost surely
    \[
        {(h_k)^{-kn}}\cdot\nu_{i}(\Arr') 
        \ \asymp \ \gamma_{i}(\Arr)
    \]
    where $\Arr' \sim \SFE{k,h_k,\ell_k}(\Arr,e_1,\dots,e_n)$, $i=jc$, and $0
    \le j \le d$. In particular, the sequence
    $(\gamma_0(\Arr),\dots,\gamma_r(\Arr))$ of Whitney numbers is log-concave.
\end{thm}
\begin{proof}
    For $n=1$ and $\Arr = \{ H \}$, the claim is immediate with the chosen
    normalization. For $n > 1$, it follows from Theorem~\ref{thm:measure1}
    that for $k \rightarrow \infty$, the Wills polynomial satisfies the same
    deletion-contraction relation~\eqref{eqn:char_del_contr} as the
    characteristic polynomial which completes the first claim. The
    log-concavity of the Whitney numbers $(\gamma_i)$ now follows from 
    Corollary~\ref{cor:wills_log}.
\end{proof}

\bibliographystyle{myamsalpha}
\bibliography{Ref}

\end{document}